\documentclass[12pt,a4paper]{amsart}
\usepackage{amsfonts,color}
\usepackage{amsthm}
\usepackage{amsmath}
\usepackage{amscd}
\usepackage[latin2]{inputenc}
\usepackage{t1enc}
\usepackage[mathscr]{eucal}
\usepackage{indentfirst}
\usepackage{graphicx}
\usepackage{graphics}
\usepackage{pict2e}
\usepackage{epic}
\usepackage{url}
\usepackage{epstopdf}
\usepackage{comment}

\numberwithin{equation}{section}
\usepackage[margin=2.6cm]{geometry}

\allowdisplaybreaks

\theoremstyle{plain}
\newtheorem{Th}{Theorem}[section]
\newtheorem{Lemma}[Th]{Lemma}

 \theoremstyle{definition}
\newtheorem{Def}[Th]{Definition}

\newtheorem{Rem}[Th]{Remark}
\newtheorem{?}[Th]{Problem}

\newcommand{\prm}{{\mathrm{pm}}}

\newcommand{\C}{\mathbb{C}}

\newcommand{\e}{\varepsilon}
\newcommand{\capp}{\mathrm{cap}}
\newcommand{\R}{\mathbb{R}}
\newcommand{\HH}{\mathbb{H}}

\begin{document}

\title[Stable polynomials, orientations and matchings]{Short survey on \\ stable polynomials, orientations and matchings}

\author[P. Csikv\'ari]{P\'{e}ter Csikv\'{a}ri}

\address{Alfr\'ed R\'enyi Institute of Mathematics, H-1053 Budapest Re\'altanoda utca 13/15 \and ELTE: E\"{o}tv\"{o}s Lor\'{a}nd University \\ Mathematics Institute, Department of Computer
Science \\ H-1117 Budapest
\\ P\'{a}zm\'{a}ny P\'{e}ter s\'{e}t\'{a}ny 1/C}

\email{peter.csikvari@gmail.com}

\author[\'A. Schweitzer]{\'Ad\'am Schweitzer}

\address{E\"{o}tv\"{o}s Lor\'{a}nd University \\ Mathematics Institute \\ H-1117 Budapest
\\ P\'{a}zm\'{a}ny P\'{e}ter s\'{e}t\'{a}ny 1/C \\ Hungary
}

\email{adamschweitzer1@gmail.com}

\thanks{The first author  is supported by the Counting in Sparse Graphs Lend\"ulet Research Group of the Alfr\'ed R\'enyi Institute of Mathematics. The second author is partially supported by the EFOP program (EFOP-3.6.3-VEKOP-16-2017-00002).}



\begin{abstract} This is a short survey about the theory of stable polynomials and its applications. It gives  self-contained proofs of two theorems of Schrijver.
One of them asserts that for a $d$--regular bipartite graph $G$ on $2n$ vertices, the number of perfect matchings, denoted by $\prm(G)$,  satisfies
$$\prm(G)\geq \bigg( \frac{(d-1)^{d-1}}{d^{d-2}} \bigg)^{n}.$$
The other theorem claims that for even $d$ the number of Eulerian orientations of a $d$--regular graph $G$ on $n$ vertices, denoted by $\e(G)$, satisfies
$$\e(G)\geq \bigg(\frac{\binom{d}{d/2}}{2^{d/2}}\bigg)^n.$$
To prove these theorems we use the theory of stable polynomials, and give a common generalization of the two theorems.
\end{abstract}

\maketitle

\section{Introduction}

In this paper we will give new proofs for the following theorems of Schrijver. The first theorem is about perfect matchings of regular bipartite graphs.

\begin{Th}[Schrijver \cite{Sch2}] \label{matching}
Let $G=(A,B,E)$ be $d$--regular bipartite graph on $2n$ vertices. Let $\prm(G)$ denote the number of perfect matchings of $G$. Then
$$\prm(G)\geq \bigg( \frac{(d-1)^{d-1}}{d^{d-2}} \bigg)^{n}.$$
\end{Th}

The next theorem is about Eulerian orientations of regular graphs. Recall that an orientation of a graph $G$ is Eulerian if the in-degree and out-degree are equal at each vertex. (In particular, the degree of a vertex must be even.)

\begin{Th}[Schrijver \cite{Sch1}] \label{orientation}
Let $G$ be a graph, where the degree of the vertex $v$ is $d_v$. Suppose that $d_v$ is even for each $v$. Let $\e(G)$ denote the number of Eulerian orientations of the graph $G$. Then
$$\e(G)\geq \prod_{v\in V(G)}\frac{\binom{d_v}{d_v/2}}{2^{d_v/2}}.$$
\end{Th}

The main goal of this paper is to show that these two theorems have a common generalization. To spell out this generalization we will count the number of orientations in a graph with prescribed in-degree sequence.

\begin{Def}
Let $\underline{r}=(r_v)_{v\in V(G)}\in \mathbb{Z}^{V(G)}$. Let $\e_{\underline{r}}(G)$ denote the number of those orientations of the graph $G$, where the in-degree of the vertex $v$ is $r_v$.
\end{Def}

Observe that if we have a $d$--regular bipartite graph $G=(A,B,E)$, then the perfect matchings are in bijection with those orientations of the graph, where the in-degree of the vertices in $A$ is $1$, and is $d-1$ in case of the vertices of $B$. Indeed, simply orient each edge of a perfect matching towards $A$, and every other edge towards $B$. Clearly, if we have such an orientation, then the edges oriented towards $A$ form a perfect matching. 
\bigskip

The following theorem might look technical, but it easily implies both  Theorem~\ref{matching} and \ref{orientation}.

\begin{Th} \label{general}
Let $G=(V,E)$ be a graph with degree $d_v$ at vertex $v$, and let 
$\e_{\underline{r}}(G)$ denote the number of orientations of the graph $G$, where the in-degree of the vertex $v$ is $r_v$.  Then
$$\e_{\underline{r}}(G)\geq \prod_{v\in V(G)}\binom{d_v}{r_v}\bigg(\frac{r_v}{d_v}\bigg)^{r_v}\bigg(\frac{d_v-r_v}{d_v}\bigg)^{d_v-r_v}\cdot \inf_{x_u>0}\frac{\prod_{(u,v)\in E(G)}(x_u+x_v)}{\prod_{u\in V(G)}x_u^{r_u}}.$$
\end{Th}

We remark that the role of the multivariate polynomial $P_G(\underline{x}):=\prod_{(u,v)\in E(G)}(x_u+x_v)$ in the theorem comes from the fact that
$$\prod_{(u,v)\in E(G)}(x_u+x_v)=\sum_{\underline{r}}\e_{\underline{r}}(G)\prod_{u\in V(G)}x_u^{r_u}.$$
In other words, this theorem is about how to give a lower bound on a coefficient of a multivariate polynomial in terms of the polynomial. It turns out that this lower bound is possible, because the polynomial $P_G(\underline{x})$ is a real stable polynomial. The definition of real stability is the following.

\begin{Def}
A multivariate polynomial $P(x_1,\dots ,x_n)$ with complex coefficients is stable if $P(z_1,\dots ,z_n)\neq 0$ whenever $\mathrm{Im}(z_i)>0$ for $i=1,\dots ,n$. A polynomial is called real stable if it is stable and its coefficients are real.
\end{Def}

Note that a univariate polynomial with real coefficients is  stable if and only if it is real-rooted. So real stability is a generalization of real-rootedness for multivariate polynomials.

We remark that Gurvits \cite{Gur1} already gave a  proof of Theorem~\ref{matching} using real stable polynomials. He used the polynomial 
$$Q(\underline{x})=\prod_{v\in B}\bigg(\sum_{u\in A}x_u\bigg)$$
in his proof. In this case the coefficient of $\prod_{u\in A}x_u$ is exactly $\prm(G)$. In fact, we will follow exactly the strategy of Gurvits.  This strategy is based on two concepts, the real stability and the capacity of a polynomial. The latter was invented by Gurvits himself. We will review these concepts in the next section.
\bigskip

\noindent \textbf{What is new in this paper?} The proof 
of Theorem~\ref{coefficient theorem} is new, the theorem itself appeared in \cite{Gur3} in a slightly different form.
The use of the polynomial $P_G(\underline{x})$ in these proofs also seems to be new, although variants of this polynomial appeared in \cite{Gur3}, but never exactly this one. Proving Theorem~\ref{orientation} via stable polynomials is also new. On the other hand, the general strategy is not new at all. In fact, one of our main goals is to advertise this theory, so this paper can be considered as a mini survey.
\bigskip

\noindent \textbf{This paper is organized as follows.} In the next section we collect the basic facts about real stable polynomials and capacity. In Section 3 we prove Theorem~\ref{general} and derive Theorem~\ref{matching} and \ref{orientation} from it. In Section 4 we collected pointers to the literature.

\section{Real stability and capacity}

In this section we review the basic properties of real stability and capacity. To keep this paper  self-contained we will prove every result that we use apart from the Hermite-Sylvester criterion.

\subsection{Stability} Recall that a multivariate polynomial $P(x_1,\dots ,x_n)$ is real stable if it has real coefficients, and $P(z_1,\dots ,z_n)\neq 0$ whenever $\mathrm{Im}(z_i)>0$ for $i=1,\dots ,n$. We have seen that real stability is a generalization of univariate real-rooted polynomials. As the next lemma shows there is a more direct connection between the two concepts.

\begin{Lemma} \label{stable-real-rooted} A multivariate polynomial $P(z_1,\dots ,z_n)\in \R[z_1,\dots ,z_n]$ is stable if and only if for all $\underline{v}=(v_1,\dots, v_n)\in \R^n$ and $\underline{u}=(u_1,\dots ,u_n)\in \R^n_{>0}$, the univariate polynomial $g(t)=P(v_1+tu_1,\dots ,v_n+tu_n)$ is real-rooted.
\end{Lemma}

\begin{proof} Let $\mathbb{H}=\{z \ |\ \mathrm{Im}(z)>0\}$. First suppose that $g(t)$ is not real rooted, then it has a root $a+bi\in \C$, where $b\neq 0$. Since the coefficients of $g$ are real, $a\pm bi$ are both zeros of $g$. So we can assume that $b>0$. But then $\mathrm{Im}(v_j+u_j(a+bi))=bu_j>0$ shows that the numbers $z_j=v_j+u_j(a+bi) \in \HH$ for $j=1,\dots ,n$ and $P(z_1,\dots ,z_n)=0$, thus $P$ is not stable.

Next suppose that $P$ is not stable, it has some zero $(z_1,\dots ,z_n)\in \HH^n$. Let $z_j=a_j+b_ji$. Then $b_j>0$. Then the polynomial
$g(t)=P(a_1+tb_1,\dots ,a_n+tb_n)$ satisfies that $(a_1,\dots ,a_n)\in \R^n$ and $(b_1,\dots ,b_n)\in \R^n_{>0}$ and has a non-real zero, namely $i$.
\end{proof}

We will also need the following lemma. 

\begin{Lemma} \label{real-rooted-convergence}
Given a polynomial $P(x)$ and suppose that $(P_n(x))_n$ is a sequence of real-rooted  polynomials such that
$\lim_{n\to \infty}P_n(x)=P(x)$ coefficientwise. Then $P(x)$ is real-rooted.
\end{Lemma}

Lemma~\ref{real-rooted-convergence} is a simple consequence of the well-known Hermite-Sylvester criterion. For a short, simple proof see that paper \cite{Nat}.

\begin{Lemma}[Hermite-Sylvester criterion] \label{HS-criterion}
Let $P(x)$ be a non-constant polynomial of degree $d$ with real coefficients and $\lambda_1,\dots ,\lambda_d$ not necessarily distinct zeros. Let $m_k=\sum_{j=1}^d\lambda_j^k$. Then $P(x)$ is real-rooted if and only if the $d\times d$ matrix $(m_{i+j-2})_{i,j=1,\dots ,d}$ is positive semi-definite. 
\end{Lemma}

Note that $m_k$ can be computed from the coefficients by the Newton-Girard identities, so they are continuous functions of the coefficients. Hence the  Hermite-Sylvester criterion implies Lemma~\ref{real-rooted-convergence}.

Next we collect some operations that preserve stability. A general theory of stability preserver operations is developed by Borcea and Br\"and\'en \cite{BoBr}. We also recommend the paper of Choe, Oxley, Sokal and Wagner \cite{COSW} for a comprehensive list of operations that preserves stability. 

\begin{Th}
Let $P(x_1,\dots ,x_n)$ be a real stable polynomial. Suppose that the degree of $x_1$ in $P$ is $d$. Then the following hold true.\\
(a) The polynomial $x_1^dP(-1/x_1,x_2,\dots x_n)$ is real stable.\\
(b) If $a\in \mathbb{R}$, then $P(a,x_2,\dots ,x_n)$ is real stable or the constant $0$ polynomial.\\
(c) The polynomial $\frac{\partial}{\partial x_1}P$ is real stable or the constant $0$ polynomial.\\
\end{Th}

\begin{proof}
As before let $\mathbb{H}=\{z \ |\ \mathrm{Im}(z)>0\}$. Then the first claim is trivial since $z\mapsto \frac{-1}{z}$ maps $\mathbb{H}$ to $\mathbb{H}$.
\medskip

Next we prove part (b). By Lemma~\ref{stable-real-rooted} the polynomial $P(a+\varepsilon t,v_2+tu_2,\dots ,v_n+tu_n)$ is real-rooted for every $\varepsilon,u_2,\dots u_n\in \R_{>0}$ and $a,v_2,\dots ,v_n\in \R$. Let $\varepsilon \to 0$, then by Lemma~\ref{real-rooted-convergence} we get that $P(a,v_2+tu_2,\dots ,v_n+tu_n)$ is real-rooted or the constant $0$ function for every $u_2,\dots ,u_n\in \R_{>0}$ and $v_2,\dots ,v_n$. Now using the  other direction of Lemma~\ref{stable-real-rooted} we get that $P(a,x_2,\dots ,x_n)$ is a real stable polynomial.

Next we prove part (c). Let 
$$P(x_1,\dots ,x_n)=\sum_{k=0}^dP_k(x_2,\dots ,x_n)x_1^k.$$
First we show that $P_d(x_2,\dots ,x_n)$ is a real stable polynomial. By part (a) 
$$R(x_1,\dots ,x_n):=x_1^dP(-1/x_1,x_2,\dots x_n)=\sum_{k=0}^dP_k(x_2,\dots ,x_n)(-1)^kx_1^{d-k}$$
is real stable. Then $R(0,x_2,\dots ,x_n)=(-1)^dP_d(x_2,\dots ,x_n)$ is real stable, and so $P_d(x_2,\dots ,x_n)$ is real stable.

Now let $Q=\frac{\partial}{\partial x_1}P$, and let $\underline{a}=(a_1,a_2,\dots, a_n)\in \mathbb{H}^n$. We show that if $Q\not\equiv 0$, then $\mathrm{Im}\left(\frac{Q(\underline{a})}{P(\underline{a})}\right)<0$. Note that here we use that $P$ is stable, so $P(\underline{a})\neq 0$, and we can divide with it. If $d=0$ then $Q\equiv 0$. We can assume that $d\geq 1$. Let 
$$g(x):=P(x,a_2,\dots ,a_n)=\sum_{k=0}^dP_k(a_2,\dots ,a_n)x^k.$$
Note that $P_d(a_2,\dots ,a_n)\neq 0$ since $P_d$ is real-stable. So $g(x)$ has degree $d\geq 1$. Then
$g(x)=c\prod_{i=1}^d(x-\rho_i)$, and we have
$$\frac{g'(x)}{g(x)}=\sum_{i=1}^d\frac{1}{x-\rho_i}.$$
Note that $\mathrm{Im}(\rho_i)\leq 0$, otherwise $P(\rho_i,a_2,\dots ,a_n)=0$ would yield a zero in $\mathbb{H}^n$. Hence 
$$\mathrm{Im}\left(\frac{Q(\underline{a})}{P(\underline{a})}\right)=\mathrm{Im} \left(\frac{g'(a_1)}{g(a_1)}\right)=\mathrm{Im} \left(
\sum_{i=1}^d\frac{1}{a_1-\rho_i}\right)<0.$$
In particular, this shows that $Q(\underline{a})\neq 0$. Hence $Q$ is stable. (Remark: we essentially repeated the proof of Gauss--Lucas theorem that asserts that the zeros of the derivative of a polynomial lie in the convex hull of the zeros of the polynomial.)
\medskip

\end{proof}

\subsection{Capacity of a polynomial} In this section we introduce the concept capacity.

\begin{Def}[Gurvits \cite{Gur1}]
Let $P(x_1,\dots ,x_n)$ be a multivariate polynomial with non-negative coefficients. Let $\underline{\alpha}=(\alpha_1,\dots ,\alpha_n)$ be a non-negative vector. Then the $\underline{\alpha}$-capacity of the polynomial $P(x_1,\dots ,x_n)$ is 
$$\mathrm{cap}_{\underline{\alpha}}(P)=\inf_{x_1,\dots ,x_n>0}\frac{P(x_1,\dots ,x_n)}{\prod_{i=1}^nx_i^{\alpha_i}}.$$ 
\end{Def}

Note that if the numbers $\alpha_i$ are integers, then the capacity is an upper bound for the coefficient of the term $\prod_{i=1}^nx_i^{\alpha_i}$.

So in Theorem~\ref{general} we had $\mathrm{cap}_{\underline{r}}(P_G)$ in the statement.
In the definition of capacity we never used that $P$ is real stable, but it turns out that this concept is especially useful when we study stable polynomials. The main reason for this phenomenon is that one can often govern the capacity for stability preserver operators. A general theory of capacity preserver linear operators was developed in the  paper of Leake and Gurvits \cite{LeGu}.   The following theorem is a special case of their theory.

\begin{Th} \label{capacity preserving}
Let $P(x_1,\dots ,x_n)$ be a real stable polynomial with non-negative coefficients. Suppose that the degree of $x_1$ in $P$ is $d$. Let
$$Q=\frac{1}{r!}\bigg( \frac{\partial^r}{\partial x_1^r}P\bigg)\bigg|_{x_1=0}.$$
In other words, if we expand $P$ as a polynomial of $x_1$, then $Q$ is the coefficient of $x_1^r$. \\
Let $\underline{\alpha}=(\alpha_1,\dots ,\alpha_n)$, where $\alpha_1=r$, and $\underline{\alpha}'=(\alpha_2,\dots ,\alpha_n)$. Then
$$\mathrm{cap}_{\underline{\alpha}'}(Q)\geq \binom{d}{r}\bigg(\frac{r}{d}\bigg)^{r}\bigg(\frac{d-r}{d}\bigg)^{d-r}\mathrm{cap}_{\underline{\alpha}}(P).$$
\end{Th}

An immediate corollary of Theorem~\ref{capacity preserving} is the following theorem.

\begin{Th}[Coefficient lemma for  stable polynomials with non-negative coefficients] \label{coefficient theorem}
Let $P(x_1,\dots ,x_n)$ be a real stable polynomial with non-negative coefficients. Suppose that the degree of $x_i$ in $P$ is at most $d_i$ for $i=1,\dots ,n$. Let $\underline{r}=(r_1,\dots ,r_n)$ and $a_{\underline{r}}$ be the coefficient of $\prod_{i=1}^nx_i^{r_i}$ in $P$. Then
$$a_{\underline{r}}\geq \prod_{i=1}^n\binom{d_i}{r_i}\bigg(\frac{r_i}{d_i}\bigg)^{r_i}\bigg(\frac{d_i-r_i}{d_i}\bigg)^{d_i-r_i}\mathrm{cap}_{\underline{r}}(P).$$
\end{Th}

If we apply Theorem~\ref{coefficient theorem} to the real stable polynomial  $P_G(\underline{x})=\prod_{(u,v)\in E(G)}(x_u+x_v)$ we get Theorem~\ref{orientation}.

The key lemma to prove the above theorems is the following.

\begin{Lemma} \label{main lemma}
Let $p(z)=\sum_{k=0}^da_kz^k$ be a real-rooted polynomial with non-negative coefficients. Then
$$a_r\geq \binom{d}{r}\bigg(\frac{r}{d}\bigg)^r\bigg(\frac{d-r}{d}\bigg)^{d-r}\inf_{t>0}\frac{p(t)}{t^r}.$$
\end{Lemma}

We remark that this lemma for $r=1$ was the main ingredient of the proof of Gurvits for Theorem~\ref{matching}.

We will derive Lemma~\ref{main lemma} from the following statement.

\begin{Lemma} \label{main lemma2}
Let $p(z)=\sum_{k=0}^da_kz^k$ be a real-rooted polynomial with non-negative coefficients. Suppose that $p(1)=1$ and $p'(1)=r$ is an integer. Then
$$a_r\geq \binom{d}{r}\bigg(\frac{r}{d}\bigg)^r\bigg(\frac{d-r}{d}\bigg)^{d-r}.$$
\end{Lemma}

We remark that Lemma~\ref{main lemma2} is a special case of the following theorem of Hoeffding \cite{Hoe}. Nevertheless to keep our paper self-contained we will give a proof of Lemma~\ref{main lemma2}. The intuitive meaning of Hoeffding's theorem is that among probability distributions coming from real-rooted polynomials and fixed expected value the binomial distribution is the least concentrated around its expected value.

\begin{Th}[Hoeffding \cite{Hoe}] \label{Hoeffding} Let $p(z)=\sum_{k=0}^dp_kz^k$ be a real-rooted polynomial with $p_k\geq 0$, and $p(1)=1$, that is, $\sum_{k=0}^dp_k=1$. Let $s$ be defined by the equation $\sum_{k=0}^dkp_k=ds$.
Suppose that for non-negative integers $b$ and $c$ we have $b\leq ds\leq c$. Then
$$\sum_{k=b}^cp_k\geq \sum_{k=b}^c\binom{d}{k}s^k(1-s)^{d-k}.$$
\end{Th}

 It is  easy to see that Lemma~\ref{main lemma2} is a special case of Lemma~\ref{main lemma}, but as the following proof shows they are actually equivalent statements.

\begin{proof}[Proof of Lemma~\ref{main lemma} from Lemma~\ref{main lemma2}.] 
Suppose that $p(z)=\sum_{k=m}^Ma_kz^k$, where $a_m,a_M>0$.
If $r<m$ or $r>M$, then $\inf_{t>0} \frac{p(t)}{t^r}=0$ so the claim is true in this case. If $r=m$, then $\inf_{t>0} \frac{p(t)}{t^r}=a_m$ so the claim is again true. If $r=M$, then  $\inf_{t>0} \frac{p(t)}{t^r}=a_M$ so we are again done. 
Thus we can assume that $m<r<M$.
Observe that $\frac{tp'(t)}{p(t)}$ is monotone increasing, $\frac{tp'(t)}{p(t)}\big|_{t=0}=m$ and $\lim_{t\to \infty} \frac{tp'(t)}{p(t)}=M$. So we can choose $t_r>0$ in such a way that $\frac{t_rp'(t_r)}{p(t_r)}=r$. Let us consider the probability distribution $q_j=\frac{a_jt_r^j}{p(t_r)}$. Then $\sum_k kq_k=r$, and $\sum_{j=0}^dq_jz^j$ is still a real-rooted polynomial.
Next let us apply Lemma~\ref{main lemma2}. Then
$$\frac{a_rt_r^r}{p(t_r)}=q_r\geq \binom{d}{r}\left(\frac{r}{d}\right)^r\left(\frac{d-r}{d}\right)^{d-r}.$$
In other words,
$$a_r\geq \binom{d}{r}\left(\frac{r}{d}\right)^r\left(\frac{d-r}{d}\right)^{d-r}\frac{p(t_r)}{t_r^r}\geq \binom{d}{r}\left(\frac{r}{d}\right)^r\left(\frac{d-r}{d}\right)^{d-r}\inf_{t>0}\frac{p(t)}{t^r}.$$
\end{proof}

Before we prove Lemma~\ref{main lemma2} we show that Lemma~\ref{main lemma} indeed implies Thorem~\ref{capacity preserving}.

\begin{proof}[Proof of Theorem~\ref{capacity preserving}]
For fixed $a_2,\dots ,a_n$ consider the polynomial 
$$g(x)=P(x,a_2,\dots ,a_{n}).$$
Clearly, $\frac{1}{r!}\frac{d^r}{dx^r}g(x)\Big|_{x=0}=Q(a_2,\dots ,a_{n})$. Since $P$ is stable, and we substituted $a_i\in \mathbb{R}$ into it, $g(x)$ is stable. In other words, it is real-rooted. Thus we can use Theorem~\ref{main lemma}:
$$a_r \geq  \binom{d}{r}\bigg(\frac{r}{d}\bigg)^r\bigg(\frac{d-r}{d}\Bigg)^{d-r}\inf_{x>0}\frac{g(x)}{x^r}.$$
Hence we have
\begin{align*}
\frac{Q(a_2,\dots ,a_{n})}{\prod_{i=2}^{n}a_i^{\alpha_i}}=\frac{\frac{1}{r!}\frac{d^r}{dx^r}g(x)\Big|_{x=0}}{\prod_{i=2}^{n}a_i^{\alpha_i}}&\geq \binom{d}{r}\bigg(\frac{r}{d}\bigg)^r\bigg(\frac{d-r}{d}\bigg)^{d-r}\frac{1}{\prod_{i=2}^{n}a_i^{\alpha_i}}\inf_{x>0}\frac{g(x)}{x^r}\\
                            &= \binom{d}{r}\bigg(\frac{r}{d}\bigg)^r\bigg(\frac{d-r}{d}\bigg)^{d-r}\inf_{x>0} \frac{P(x,a_2,\dots ,a_{n})}{x^r\cdot \prod_{i=2}^{n} a_i^{\alpha_i}}\\
							 &\geq \binom{d}{r}\bigg(\frac{r}{d}\bigg)^r\bigg(\frac{d-r}{d}\bigg)^{d-r}\mathrm{cap}_{\underline{\alpha}}(P).
\end{align*}
Taking infimum on the left side we get that
$\mathrm{cap}_{\underline{\alpha}'}(Q)\geq \binom{d}{r}\big(\frac{r}{d}\big)^r\big(\frac{d-r}{d}\big)^{d-r}\mathrm{cap}_{\underline{\alpha}}(P)$.
\end{proof}

\subsection{Proof of Lemma~\ref{main lemma2}}
In this section we prove Lemma~\ref{main lemma2}. The condition on the non-negativity of the coefficients and $p(1)=1$ implies that $p(z)$ can be written as follows:
$$p(z)=\prod_{i=1}^d(1-\alpha_i+\alpha_iz),$$
where $0\leq \alpha_i\leq 1$. Indeed, as all coefficients are non-negative, there can be no positive roots, thus using $p(1)=1$ the polynomial can be rewritten in the following way:
$$p(z)=a_n\prod_{i=1}^d(z+\rho_i)=\prod_{i=1}^d\left(\frac{z+\rho_i}{1+\rho_i}\right)=\prod_{i=1}^d(1-\alpha_i+\alpha_iz),$$
where $\alpha_i=\frac{1}{1+\rho_i}$.
Note that
$$r=p'(1)=\sum_{i=1}^d \alpha_i\prod_{j\not=i}(\alpha_j+1-\alpha_j)=\sum_{i=1}^d \alpha_i.$$
Consider the domain 
$$D_{d,r}=\left\{(\alpha_1,\dots ,\alpha_d)\in \mathbb{R}^d\ \Big|\ 0\leq \alpha_i\leq 1\ (i=1,\dots ,d),\ \sum_{i=1}^d\alpha_i=r\right\}.$$
Clearly, the coefficient $a_r$ of $z^r$ in $p(z)$ can be expressed as
$$a_{r}=\sum_{\substack{K\subset[n]\\ |K|=r}}\prod_{j\in K}\alpha_j\prod_{j\not\in K}(1-\alpha_j).$$
So let us introduce the function
$$f_{d,r}(x_1,\ldots,x_d)=\displaystyle\sum_{\substack{K\subset[n]\\ |K|=r}}\prod_{j\in K}x_j\prod_{j\not\in K}(1-x_j).$$
Clearly, the statement of Lemma~\ref{main lemma2} is equivalent with
$$\min_{\underline{x}\in D_{d,r}}f_{d,r}(\underline{x})=f_{d,r}\left(\frac{r}{d},\dots ,\frac{r}{d}\right)=\binom{d}{r}\bigg(\frac{r}{d}\bigg)^r\bigg(\frac{d-r}{d}\bigg)^{d-r}.$$
We will prove this statement by induction on $d$. The case $d=1$ and $r=0$ or $r=1$ is trivial. In general, the case $r=0$ or $r=d$ is trivial since $D_{d,r}$ consists of only one point in this case. So we can always assume that $0<r<d$. First we prove that the statement is true for the boundary $\partial D_{d,r}$. Then we will prove that  if $\underline{x}\neq \left(\frac{r}{d},\dots ,\frac{r}{d}\right)$, then we either have a point $\underline{y}\in \partial D_{d,r}$ for which $f_{d,r}(\underline{x})\geq f_{d,r}(\underline{y})$ or there exists an $\underline{x}'\in D_{d,r}$ such that $f_{d,r}(\underline{x})> f_{d,r}(\underline{x}')$.
 The compactness of $D_{d,r}$ then implies that $\min_{\underline{x}\in D_{d,r}}f_{d,r}(\underline{x})=f_{d,r}\left(\frac{r}{d},\dots ,\frac{r}{d}\right)$.

So let us first prove that for $\underline{x}\in \partial D_{d,r}$ we have $f_{d,r}(\underline{x})\geq f_{d,r}\left(\frac{r}{d},\dots ,\frac{r}{d}\right)$.
Clearly, if $\underline{x}\in \partial D_{d,r}$ then one of its coordinates is $0$ or $1$. If we delete this coordinate, then the obtained vector $\underline{x}'$ is in $D_{d-1,r}$ in the first case, and in $D_{d-1,r-1}$ in the second case. Furthermore, $f_{d,r}(\underline{x})=f_{d-1,r}(\underline{x}')$ in the first case, and $f_{d,r}(\underline{x})=f_{d-1,r-1}(\underline{x}')$ in the second case. By induction we know that we have
$$f_{d-1,r}(\underline{x}')\geq \binom{d-1}{r}\bigg(\frac{r}{d-1}\bigg)^r\bigg(\frac{d-1-r}{d-1}\bigg)^{d-1-r}$$
for $\underline{x}'\in D_{d-1,r}$, and 
$$f_{d-1,r-1}(\underline{x}')\geq \binom{d-1}{r-1}\bigg(\frac{r-1}{d-1}\bigg)^{r-1}\bigg(\frac{d-r}{d-1}\bigg)^{d-r}$$
for $\underline{x}'\in D_{d-1,r-1}$. Let us introduce the function $\ell_{n,k}=x^k(1-x)^{n-k}$. It is easy to see that it takes its maximum at the value $x=k/n$ in the interval $[0,1]$ as its derivative is $(k-nx)x^{k-1}(1-x)^{n-k-1}$. Hence
\begin{align*}
\binom{d}{r}\bigg(\frac{r}{d}\bigg)^r\bigg(\frac{d-r}{d}\bigg)^{d-r}&=\binom{d-1}{r}\bigg(\frac{r}{d}\bigg)^r\bigg(\frac{d-r}{d}\bigg)^{d-r-1}\\
&=\binom{d-1}{r}\ell_{d-1,r}\left(\frac{r}{d}\right)\\
&<\binom{d-1}{r}\ell_{d-1,r}\left(\frac{r}{d-1}\right)\\
&=\binom{d-1}{r}\bigg(\frac{r}{d-1}\bigg)^r\bigg(\frac{d-1-r}{d-1}\bigg)^{d-1-r}
\end{align*}
and
\begin{align*}
\binom{d}{r}\bigg(\frac{r}{d}\bigg)^r\bigg(\frac{d-r}{d}\bigg)^{d-r}&=\binom{d-1}{r-1}\bigg(\frac{r}{d}\bigg)^{r-1}\bigg(\frac{d-r}{d}\bigg)^{d-r}\\
&=\binom{d-1}{r-1}\ell_{d-1,r-1}\left(\frac{r}{d}\right)\\
&<\binom{d-1}{r-1}\ell_{d-1,r-1}\left(\frac{r-1}{d-1}\right)\\
&=\binom{d-1}{r-1}\bigg(\frac{r-1}{d-1}\bigg)^{r-1}\bigg(\frac{d-r}{d-1}\bigg)^{d-r}.
\end{align*}
Hence for $\underline{x}\in \partial D_{d,r}$ we have $f_{d,r}(\underline{x})\geq f_{d,r}\left(\frac{r}{d},\dots ,\frac{r}{d}\right)$. 

Next we show that if $\underline{x}\neq \left(\frac{r}{d},\dots ,\frac{r}{d}\right)$, then either 
we have a point $\underline{y}\in \partial D_{d,r}$ for which $f_{d,r}(\underline{x})\geq f_{d,r}(\underline{y})$ or
there exists an $\underline{x}'\in D_{d,r}$ such that $f_{d,r}(\underline{x})> f_{d,r}(\underline{x}')$. Since $\underline{x}\neq \left(\frac{r}{d},\dots ,\frac{r}{d}\right)$ there exists $i$ and $j$ such that $x_i\neq x_j$. Let $x_i+x_j=u$ and 
$$\prod_{k\neq i,j}(1-x_k+x_kz)=\sum_{m=0}^{d-2}b_mz^m.$$
Then $f_{d,r}(\underline{x})$ is the coefficient of $z^r$ in the polynomial 
$$\left(\sum_{m=0}^{d-2}b_mz^m\right)(1-x_i+zx_i)(1-(u-x_i)+z(u-x_i)).$$
A little computation shows that
$$f_{d,r}(\underline{x})=x_i(u-x_i)(b_{r-2}-2b_{r-1}+b_r)+(rb_{r-1}+(1-r)b_r).$$
If $b_{r-2}-2b_{r-1}+b_r<0$, then we get a strictly smaller value for $f_{d,r}(\underline{x})$ if we replace $(x_i,x_j)$ with $(u/2,u/2)$. If $b_{r-2}-2b_{r-1}+b_r\geq 0$, then we can replace $(x_i,x_j)$ with $(0,u)$ or $(1,u-1)$ depending on $u\in [0,1]$ or $[1,2]$ yielding a boundary point $\underline{y}$ for which $f_{d,r}(\underline{x})\geq f_{d,r}(\underline{y})$. This completes the proof.

\section{Capacity of the polynomial $P_G$}

In this section we prove Theorems~\ref{matching} and \ref{orientation} by computing the capacity of $P_G(\underline{x})=\prod_{(u,v)\in E}(x_u+x_v)$ with respect to various vectors $\underline{\alpha}$.

\begin{Lemma} \label{lemma-matching} Let $G=(A,B,E)$ be a $d$--regular bipartite graph on $2n$ vertices. Let $\underline{\alpha}$ be the vector that takes value $1$ at a vertex $u\in A$, and value $d-1$ at a vertex $v\in B$. Then
$$\capp_{\underline{\alpha}}(P_G) = \frac{d^{nd}}{(d-1)^{n(d-1)}}.$$
\end{Lemma}

\begin{proof} For sake of convenience let us denote the variables by $x_u$ if $u\in A$, and $y_v$ if $v\in B$. Then
$$\capp_{\underline{\alpha}} (P_G)=\inf_{x_u,y_v>0 \atop u\in A, v\in B} \frac{ \prod_{(u,v)\in E(G)}(x_u+y_v)}{\prod_{u\in A}x_u \cdot  \prod_{v\in B}y_v^{d-1}}=\inf_{x_u,y_v>0 \atop u\in A, v\in B}\prod_{(u,v)\in E(G)}\frac{x_u+y_v}{x_u^{1/d}y_v^{(d-1)/d}}.$$
Note that
$$x_u+y_v=\frac{1}{d}(dx_u)+\frac{d-1}{d}\left(\frac{dy_v}{d-1}\right)\geq (dx_u)^{1/d}\left(\frac{dy_v}{d-1}\right)^{(d-1)/d}$$
by weighted arithmetic-geometric mean inequality. In other words,
$$\frac{x_u+y_v}{x_u^{1/d}y_v^{(d-1)/d}}\geq \frac{d}{(d-1)^{(d-1)/d}}.$$
Hence
$$\capp_{\underline{\alpha}} (P_G)\geq \left(\frac{d}{(d-1)^{(d-1)/d}}\right)^{nd}=\frac{d^{nd}}{(d-1)^{n(d-1)}}.$$
Observe that if $x_u=1$ and $y_v=d-1$ for all $u\in A$ and $v\in B$ then this bound is sharp.
\end{proof}

\begin{proof}[Proof of Theorem~\ref{matching}]
As before let  $\underline{\alpha}$ be the vector that takes value $1$ at a vertex $u\in A$, and values $d-1$ at a vertex $v\in B$.
By Theorem~\ref{general} we have
$$\prm(G)\geq  \prod_{v\in A}\binom{d}{1}\bigg(\frac{1}{d}\bigg)^{1}\bigg(\frac{d-1}{d}\bigg)^{d-1}\cdot \prod_{v\in B}\binom{d}{d-1}\bigg(\frac{d-1}{d}\bigg)^{d-1}\bigg(\frac{1}{d}\bigg)^{1}\cdot \mathrm{cap}_{\underline{\alpha}}(P_G).$$
Then by Lemma~\ref{lemma-matching} we have
$$\prm(G)\geq \left(d\bigg(\frac{d-1}{d}\bigg)^{d-1}\right)^{2n}\frac{d^{nd}}{(d-1)^{n(d-1)}}=\bigg( \frac{(d-1)^{d-1}}{d^{d-2}} \bigg)^{n}.$$
\end{proof}

\begin{Lemma} \label{lemma-orientation} Let $\underline{\alpha}=(d_1/2,\dots ,d_n/2)$, where $d_j$ is the degree of the vertex $j$. Then
$\capp_{\underline{\alpha}} (P_G) = 2^{e(G)}$,
where  $e(G)$ is the number of edges. 
\end{Lemma}

\begin{proof}
First we will prove that $\capp_{\underline{\alpha}}(P_G) \leq 2^{e(G)}$.
Indeed, if we substitute $\underline{1}$ as $\underline{x}$, then we get
$$\frac{P_G(1,1, \dots 1)}{\prod_{v\in V(G)}1^{d_v/2}}=\prod_{(i,j)\in E(G)}(1+1)=2^{e(G)}.$$
Next we will prove the other direction: $\capp _{\underline{\alpha}}(P_G) \geq 2^{e(G)}$.
Using that $x_i+x_j\geq 2\sqrt{x_ix_j}$ we get that
$$\frac{\prod_{(u,v)\in E}(x_u+x_v)}{\prod_{v\in V(G)}x_v^{d_v/2}} \geq \frac{\prod_{(i,j)\in E(G)} 2\sqrt{x_ix_j}}{\prod_{v\in V(G)}x_v^{d_v/2}}=2^{e(G)}.$$
Thus $\capp_{\underline{\alpha}} (P_G) = 2^{e(G)}$.
\end{proof}

\begin{proof}[Proof of Theorem~\ref{orientation}]
As before let $\underline{\alpha}=(d_1/2,\dots ,d_n/2)$, where $d_j$ is the degree of the vertex $j$. 
By Theorem~\ref{general} we have
$$\e(G)\geq  \prod_{v\in V}\binom{d_v}{d_v/2}\bigg(\frac{d_v/2}{d_v}\bigg)^{d_v/2}\bigg(\frac{d_v/2}{d_v}\bigg)^{d_v/2}\cdot \mathrm{cap}_{\underline{\alpha}}(P_G).$$
Then by Lemma~\ref{lemma-orientation} we have
$$\e(G)\geq \prod_{v\in V}\frac{\binom{d_v}{d_v/2}}{2^{d_v}}\cdot 2^{e(G)}=\prod_{v\in V}\frac{\binom{d_v}{d_v/2}}{2^{d_v/2}}.$$
\end{proof}

\begin{Rem}
For $d$--regular graphs Las Vergnas \cite{Ver1} improved Theorem~\ref{orientation} as follows:
$$\e(G)\geq \frac{2^d}{\binom{d}{d/2}}\left(\frac{\binom{d}{d/2}}{2^{d/2}}\right)^n.$$
(At Proposition 5.3 of \cite{Ver1} there is a typo as it is pointed out in another paper of Las Vergnas  in the footnote of the first page of \cite{Ver2}.) This strengthening can be obtained by our method too: all we have to note that we only need to apply Theorem~\ref{capacity preserving} to $n-1$ variables corresponding to vertices of $P_G(\underline{x})$: at the very end we should get a polynomial of the form $cx_n^{d_n/2}$ since throughout the process we get homogeneous polynomials. For this univariate polynomial there is no need to apply Theorem~\ref{capacity preserving} once more, so we get a $2^d/\binom{d}{d/2}$ improvement. This argument applies to non-regular graphs too thereby saving a factor $2^{\Delta}/\binom{\Delta}{\Delta/2}$, where $\Delta$ is the largest degree.

\end{Rem}

\section{Beyond this paper}

In this section we collected some pointers to the literature.

Stable polynomials have a huge literature.
If someone is interested in a  comprehensive introduction to the this theory, then the paper of Choe, Oxley, Sokal and Wagner \cite{COSW} or Wagner's survey \cite{Wag} might be a good choice. Another excellent survey of the area is the paper of Vishnoi \cite{Vis}.  For capacity preserver operations the paper of  Gurvits and Leake \cite{LeGu} gives a treatment that is both very general and very readable. Statements on capacity often boils down to some statement about coefficients of univariate real-rooted polynomials often with a probabilistic flavour like Hoeffding's theorem \cite{Hoe}. A good source of such inequalities and results is Pitman's survey \cite{Pit}. The coefficient lemma, Theorem~\ref{coefficient theorem}, already appeared in the paper \cite{Gur3}. Interestingly this paper also considers various versions of the polynomial $P_G(\underline{x})=\prod_{(u,v)\in E(G)}(x_u+x_v)$, but never exactly this form. The coefficient lemma could have been easily derived from the work of  Gurvits and Leake \cite{LeGu} too.

There are many different proofs and generalizations of Theorem~\ref{matching}. The original proof of Schrijver \cite{Sch2} is elementary, but involved. The first proof based on stable polynomials and capacity is due to Gurvits \cite{Gur1}, his proof is simplified in the paper of Laurent and Schrijver \cite{LaSc}. Another proof based on the theory of graph covers is given by Csikv\'ari \cite{Csik}. (The relationship between the theory of graph covers and the theory of stable polynomials is not yet well-understood.) Theorem~\ref{matching} has a very natural generalization for permanents of non-negative matrices. This generalization was derived by Gurvits \cite{Gur2} from the original paper of Schrijver \cite{Sch2}. Subsequently, Anari and Oveis-Gharan \cite{A-OG} and Straszak and Vishnoi \cite{StVi} gave a proof that only relies on the theory of stable polynomials. Another possible generalization considers counting matchings of fixed size in bipartite graphs instead of perfect matchings. This question was treated in the papers Csikv\'ari \cite{Csik}, Lelarge \cite{Lela} and Gurvits and Leake \cite{LeGu} (this last one uses only stable polynomials). 

Concerning Eulerian orientations, Theorem~\ref{orientation} has many different proofs either. The original proof of Schrijver is very elegant and simple. Las Vergnas \cite{Ver1}  gave another proof building on the theory of Martin's polynomial that gives a slightly stronger result. Borb\'enyi and Csikv\'ari \cite{BoCs} gave a proof using gauge transformation. The proof presented here is the first one using stable polynomials, but we remark that this result could have been easliy deduced from the paper of Straszak and Vishnoi \cite{StVi} too that uses stable polynomials. 

Both Theorems~\ref{matching} and \ref{orientation} can be interpreted as a correlation inequality. If we divide by $2^{e(G)}=\prod_{v\in V(G)} 2^{d_v/2}$ in Theorem~\ref{orientation}, then the the right hand side of the inequality
$$\frac{\e(G)}{2^{e(G)}}\geq \prod_{v\in V(G)}\frac{\binom{d_v}{d_v/2}}{2^{d_v}}$$
is the probability that a random orientation is Eulerian, while on the left hand side the term $\frac{\binom{d_v}{d_v/2}}{2^{d_v}}$ is the probability that a random orientation is balanced at vertex $v$. So this inequality can be interpreted as a positive correlation inequality. Similarly, if $G=(A,B,E)$ is a $d$--regular bipartite graph, then we can consider the probability space where for each vertex $u$ in $A$ we pick exactly one of the edges incident to $u$ uniformly at random. This way we picked $n$ edges.  For a vertex $v\in B$ let $E_v$ be the event that we picked exactly one of the edges incident to $v$. Then $\mathbb{P}(\cap_{v\in B}E_v)=\frac{\prm(G)}{d^n}$ while $\mathbb{P}(E_v)=d\cdot \frac{1}{d}\left(1-\frac{1}{d}\right)^{d-1}=\left(\frac{d-1}{d}\right)^{d-1}$. Hence Theorem~\ref{matching} is equivalent with $\mathbb{P}(\cap_{v\in B}E_v)\geq \prod_{v\in B}\mathbb{P}(E_v)$. It seems that many results on capacity are indeed fuelled by positive correlation inequalities, often in a very disguised way. This connection is the most explicit in the paper \cite{StVi} where the authors introduce the iterated positive correlation property and connect it with stable polynomials. Surprisingly it is also possible to build out a theory of negative correlation based on stable polynomials, for details see the paper  \cite{Pem}. 

One might wonder whether there is a deeper connection between Theorems~\ref{matching} and \ref{orientation}. It turns out that both theorems fall into a pattern that is about the so-called Bethe approximation. Hans Bethe  was a Nobel-prize laurate physicist. Among many other achievements he introduced the concept that is now known as Bethe--approximation. Originally this was a highly heuristic concept that approximates well quantities coming from counting objects with local constraints like perfect matchings and Eulerian orientations. For a long time it was overlooked by the mathematics community, then Bethe approximation showed up in two different lines of research. In the work of Dembo, Montanari and their coauthors \cite{DeMa,DeMa2,DMS,DMSS} about graph limit theory of sparse graphs, an appropriate version of Bethe approximation played the role of the limit value of certain graph parameters. Another line of research emerged from the work of Gurvits establishing inequalities between a graph parameter and its Bethe approximation. In many cases Bethe approxiation turns out to be a lower bound for the corresponding graph parameter. Theorem~\ref{matching} and \ref{orientation} belong to this line of research. This area almost exclusively relies on stable polynomials \cite{A-OG,StVi} and graph covers \cite{Csik,Ruo,Lela,Von}.
\bigskip

\noindent \textbf{Acknowledgment.} The first author thanks Jonathan Leake for the discussions on the topic of this paper.
The authors are very grateful to the anonymous referee for his/her suggestions that greatly improved the paper.

\end{document}